\documentclass[a4paper,oneside,11pt]{article}
\usepackage[T1]{fontenc}
\usepackage[bookmarksnumbered=true]{hyperref}
\usepackage{amsfonts,amstext,amsmath,amsthm,amssymb}
\usepackage{fullpage}
\usepackage{color}
\usepackage{graphicx}
\usepackage{verbatim}
\usepackage{lipsum}
\usepackage{float} 

\usepackage{paralist}
\newtheoremstyle{slplain}
  {1.0\baselineskip\@plus.2\baselineskip\@minus.2\baselineskip}
  {.0\baselineskip\@plus.2\baselineskip\@minus.2\baselineskip}
  {\slshape}
  {}
  {\bfseries}
  {}
  { }
  {}

\usepackage[normalem]{ulem}

\newtheoremstyle{j}%
{3pt}%
{3pt}%
{}%
{\parindent}%
{\bfseries}%
{.}%
{.5em}%
{}%

\theoremstyle{plain}

\newtheorem*{rem*}{Remark}
\newtheorem*{concl*}{Conclusion}
\newtheorem*{theorem*}{Theorem}
\newtheorem*{cor*}{Corollary}
\newtheorem*{algo*}{Algorithm}

\newtheorem{theorem}{Theorem}[section]
\newtheorem{lem}[theorem]{Lemma}
\newtheorem{lemma}[theorem]{Lemma}

\newtheorem{prop}[theorem]{Proposition}
\newtheorem{deff}[theorem]{Definition}
\newtheorem{rem}[theorem]{Remark}

\theoremstyle{definition}

\newtheorem*{example*}{Example}

\newtheorem*{prob*}{Problem}

\newcommand{\Hil}{\mathcal{H}}

\newcommand{\Scal}{\mathcal{S}}

\newcommand{\Z}{\mathbb{Z}}
\newcommand{\N}{\mathbb{N}}
\newcommand{\Q}{\mathbb{Q}}
\newcommand{\R}{\mathbb{R}}
\newcommand{\Rtwo}{\mathbb{R}^2}
\newcommand{\C}{\mathbb{C}}

\newcommand{\psihat}{\widehat{\psi}}
\newcommand{\psitilde}{\widetilde{\psi}}

\newcommand{\phihat}{\widehat{\phi}}

\newcommand{\Ltwo}{L^2(\R)}

\newcommand{\ti}{\textit}

\DeclareMathOperator{\spann}{span \,}

\DeclareMathOperator{\suppp}{supp \,}

\numberwithin{equation}{section}

\usepackage{lipsum}
\usepackage{hyperref}
\usepackage{todonotes}

\begin{document}

\title{Linear independence of compactly supported separable shearlet systems}
\author{
  Jackie Ma$^{*}$
  \and
  Philipp Petersen\footnote{Technische Universit\"{a}t Berlin, 
  Department of Mathematics, 
 Stra\ss{}e des 17. Juni 136, 
10623 Berlin}
\thanks{DFG Collaborative Research Center TRR 109 "Discretization in Geometry and Dynamics"} }
\date{}
\maketitle

\begin{abstract}
This paper examines linear independence of shearlet systems. This property has already been studied for wavelets and other systems such as, for instance, for Gabor systems. In fact, for Gabor systems this problem is commonly known as the HRT conjecture. In this paper we present a proof of linear independence of compactly supported separable shearlet systems. For this, we employ a sampling strategy to utilize the structure of an implicitly given underlying oversampled wavelet system as well as the shape of the supports of the shearlet elements.
\end{abstract}

\section{Introduction}

\emph{Shearlet systems} are representation systems that were first introduced by K. Guo, G. Kutyniok, D. Labate, W.-Q Lim and G. Weiss in \cite{GuoKutLab2006, GuoLabLimWeiWil, LLKW2007}. Furthermore, compactly supported separable shearlet systems were introduced in \cite{LimShearlets, KitKutLim}, where it was also proven that these systems can constitute frames, cf. \cite{Chr}. In this paper, we study a further structural property of compactly supported separable shearlet systems, namely \emph{linear independence}. The term "linear independence" has to be clarified for this, since in an infinite dimensional space different notions are possible.
\begin{deff}\label{def:linInd}
Let $\{f_i\}_{i\in I}$ be a countable sequence of elements in a Banach space $\mathcal{X}$.
\begin{itemize}
  \item[i)] If $\sum_{i\in I} c_i f_i =0$ implies $c_i = 0$ for every $i \in I$, then we call $\{f_i\}_{i\in I}$ \emph{$\omega$-independent}.
  \item[ii)] If for any finite set $J \subset I$ we have $\sum_{i\in J} c_i f_i = 0$ if and only if $c_i  = 0$ for all $i\in J$, then we call $\{f_i\}_{i\in I}$ \emph{linearly independent} (or \emph{finitely linearly independent}).
 \end{itemize}
\end{deff}
Note that $\omega$-independence implies linear independence. The question whether certain representation systems are $\omega$-independent or linearly independent, are connected to deep conjectures in harmonic analysis, e.g. the \emph{HRT conjecture} and the \emph{Feichtinger conjecture}. We first explain Feichtinger's conjecture, whereas the HRT conjecture, formulated by C. Heil, J. Ramanathan, and P. Topiwala, will be described in Subsection \ref{subsubsec:gabor}. 

The Feichtinger conjecture, see \cite{FeichtingerConjCasazza}, claims that every bounded frame, i.e. a frame that additionally satisfies $0< \inf_{i\in I} \|f_i\|_\mathcal{H} \leq \sup_{i\in I} \|f_i\|_\mathcal{H} <\infty$, can be split into finitely many Riesz sequences, i.e. sequences $(f_i)_{i\in J}, J \subset I$ so that there exist $0<A_J \leq B_J < \infty$ such that for all $(c_i)_{i\in J} \in \ell^2(J)$ we have
\begin{align}
 A_J \|(c_i)_{i\in J}\|_{\ell^2}^2 \leq \left\|\sum \limits_{i\in J} c_i f_i \right\|_\mathcal{H}^2 \leq B_J \|(c_i)_{i\in J}\|_{\ell^2}^2.
\end{align}
In \cite{FeichtingerConjCasazza} the Feichtinger conjecture was proven to be equivalent to the \emph{Kadison Singer conjecture}, which in turn has recently been proven by A. Marcus, D. A. Spielman, and N. Srivastava by showing the \emph{paving conjecture}, see \cite{KadisonSingerProof}. 

In the next subsection we review some related work in the context of linear independence.

\subsection{Related work}
One of the first representation systems used in signal- and image processing are \emph{Gabor systems}. In this context the question whether the underlying system is linearly independent was posed. To find a general answer to this problem is, however, highly involved and leads to the still open HRT conjecture. The same question was then asked for other representation systems, such as wavelet systems and also localized frames.

\subsubsection{Gabor systems}\label{subsubsec:gabor}
Gabor analysis is build upon time-frequency shifts of a \emph{window function} $g\in L^2(\R)$ defined as
$$\pi(x,\omega)g : = e^{2\pi i \omega t} g(t-x), \quad (x,\omega) \in \R\times \hat{\R}.$$
For a subset $\Lambda \subset \R^2$ the Gabor system $\mathcal{G}(g,\Lambda)$ is defined as $\mathcal{G}(g,\Lambda):= \{ \pi(x,\omega)g, (x, \omega) \in \Lambda\}$, see \cite{Gabor, Gro2001}. If $\Lambda$ is chosen as a countable subset that is "dense enough", the theory of Feichtinger and Gr\"ochenig \cite{FG1, FG2} ensures that $\mathcal{G}(g, \Lambda)$ yields a frame for $L^2(\R)$.

It has been conjectured in \cite{HRT}, that for every non-zero function $g\in L^2(\R)$ and any set of finitely many distinct points $\left( \alpha_k, \beta_k \right )_{k=1}^N$ in $\R^2$ the set of functions $\{e^{2\pi \beta_k \cdot} g(\cdot - \alpha_k) \, : \, k=1, \ldots, N\}$ is linearly independent. This conjecture is also called \emph{HRT conjecture}. 

While the general claim remains open, there has been a lot of progress in proving the HRT conjecture for a variety of sets $\Lambda$ and functions $g$, see also the expository paper \cite{hrtnotes} and the references therein. For instance, the case where $\Lambda$ is a \emph{lattice} has been studied. As defined in \cite{HRT} a lattice in $\R^2$ is any rigid translation of a discrete subgroup of $\R^2$ generated by two linearly independent vectors in $\R^2$. It is a \emph{unit lattice} if every fundamental tile has area $1$. A result from \cite{HRT} states that, if $\Lambda$ is sampled from a unit lattice, then the Gabor system is linearly independent. For a general lattice the Theorem of Linnell \cite{Linnell} guarantees linear independence.

\subsubsection{Wavelet systems}
A wavelet system is an affine system that is build upon isotropic dilation and translation of generators, so-called \emph{mother wavelets}. We will give a more detailed introduction in Subsection \ref{subsec:wavelets}. 

Depending on the generator and the sampling of the parameters, wavelet systems can be constructed such that they constitute a frame, a Riesz basis, or even an orthonormal basis, see \cite{Dau}. 
The decomposition of wavelet frames into linearly independent subsystems has been studied, for instance, in \cite{LindnerChristensen}. In particular, the authors showed that -- in the spirit of the Feichtinger conjecture -- wavelet systems with piecewise continuous and compactly supported generators can be decomposed into finitely many linearly independent sets. Furthermore,  M. Bownik and D. Speegle showed in \cite{BownikSpeegle2010} that wavelet systems, for which the space of negative dilates is shift invariant, are linearly independent.

\subsubsection{Localized frames}

Strong localization properties of the \emph{Gramian} of a frame, such as \emph{diagonal dominance}, yield boundedness from below and hence deliver a lower Riesz bound, which in turn implies linear independence. More precisely, one can interpret the Gramian as a bi-infinite matrix and then study the localization property by analyzing the off-diagonal decay. Originally, K. Gr\"ochenig \cite{Gro} decomposed the Gramian of localized frames into diagonally dominant sub-matrices in order to extract Riesz sequences. 
While wavelet frames do not necessarily admit the necessary localization properties of the atoms, much work has been done in this direction to improve the techniques and obtain a decomposition of localized wavelet frames into Riesz sequences, see, e.g \cite{BowSpe}. 

\subsection{Our contribution}

In \cite{LimShearlets} compactly supported separable shearlet systems were constructed. In this work we discuss the question which properties of Definition \ref{def:linInd} these systems have. More precisely, we show the linear independence of a class of compactly supported shearlet systems. The precise definitions and constructions of these systems will be provided in Subsection \ref{sec:Shearlets}. 

Since $\omega$-independence implies linear independence, it is natural to investigate whether this stronger property can hold for compactly supported separable shearlets as well.  While we show linear independence for a class of shearlet systems, we will remark in Section \ref{sec:OmegaInd} that $\omega$-independence cannot hold for all such shearlet systems. 

\subsection{Outline}
The paper is structured as follows. In Section \ref{sec:prelims} we recap the necessary notations of a multiresolution analysis and a cone-adapted discrete shearlet system as introduced in \cite{LimShearlets, KitKutLim}. In particular, in Subsection \ref{sec:Shearlets} we define the compactly supported separable shearlet system that are used throughout this paper. To show the linear independence of these shearlet system, we proceed by first proving some auxiliary results for oversampled wavelet systems in Section \ref{subsec1:mainResults}. The proof of the main result (Theorem \ref{thm:main}) can then be found in in Section \ref{subsec2:mainResults}. Finally, we discuss further linear independence properties such as $\omega$-independence in Section \ref{sec:OmegaInd}.

\section{Preliminaries}\label{sec:prelims}

In this section, we give a brief presentation of the notation that we will use throughout this paper and state some preliminary results.

We will denote by $\langle \cdot,\cdot\rangle$ the standard inner product on the space of square integrable functions $L^2(\R^n)$.

\subsection{Wavelets}\label{subsec:wavelets}

We first recall some basics from wavelet theory that are needed for the rest of this paper. For a more detailed presentation of wavelets we recommend the books by I. Daubechies \cite{Dau} and E. Hern\'{a}ndez and G. Weiss \cite{Weissfirstcourse}. 

Standard wavelet systems are constructed by dilations and translations of a generating function $\psi^1 \in L^2(\R)$. For the dilated and translated versions of $\psi^1$ we write
\begin{align*}
 \psi^1_{j,m} = 2^{j/2} \psi^1(2^j \cdot - m), \quad j, m \in \Z.
\end{align*}
These systems can yield frames, Riesz bases, or even orthonormal bases under certain assumptions, see \cite{Dau}. One particular method to obtain wavelet orthonormal bases is that of a so-called \emph{multiresolution analysis} (MRA) approach. 
\begin{deff}[\cite{Dau}]\label{def:MRA}
A sequence of closed subspaces $(V_j)_{j\in \Z} \subset \Ltwo$ is called a \emph{multiresolution analysis} (MRA), if the sequence of subspaces satisfies the following properties:
\begin{compactenum}[i)]
	\item the spaces are nested, i.e. $V_j \subset V_{j+1}$ for all $j \in \Z$,
	\item the sequence is dense in $\Ltwo$ and has a trivial intersection, more precisely 
	\begin{align*}
		\overline{\bigcup \limits_{j \in \Z}} V_j = \Ltwo \quad \text{and} \quad \bigcap \limits_{j \in \Z} V_j = \{ 0 \},
	\end{align*}
	\item we have $f \in V_j$ if and only if $f(2 \cdot) \in V_{j+1}$,
	\item there exists a function $\phi^1 \in \Ltwo$, such that
	\begin{align*}
		\{ \phi^1( \cdot -m ) \, : \, m \in \Z \}
	\end{align*}
	is an orthonormal basis for $V_0$.
\end{compactenum}
The generating function $\phi^1$ in iv) is called \emph{scaling function}.
\end{deff}
Since the \ti{scaling spaces} $(V_j)_{j \in \Z}$ are closed subspaces of $L^2(\R)$ and $V_j$ is a closed subspace of $V_{j+1}$, we can define the orthogonal complement of $V_j$ in $V_{j+1}$, which we denote by $W_j$ and is called the corresponding \ti{wavelet space} at $j$-th level. In particular, we obtain an orthogonal decomposition of $V_{j+1}$ as
\begin{align*}
	V_{j+1}= V_j \oplus W_j, \quad j\in \Z. 
\end{align*}
In particular, for a given MRA $(V_j)_{j \in \Z}$, there exists an associated orthonormal basis $\{ \psi^1_{j,m} \, : \, j,m \in \Z\}$ for $\Ltwo$, such that
\begin{align*}
	P_{j+1} = P_j + \sum \limits_{m \in \Z} \langle \cdot, \psi^1_{j,m} \rangle \psi^1_{j,m}, \text{ for all } j\in \Z,
\end{align*}
where $P_j$ denotes the orthogonal projection onto $V_j$ (see \cite[Theorem 5.1.1]{Dau}). If $\phi^1$ is the generating scaling function for the MRA $(V_j)_{j \in \Z}$, we call the function $\psi^1$ a \ti{corresponding wavelet to $\phi^1$}. Note that by Definition \ref{def:MRA} ii) this yields $\bigoplus \limits_{j \in \Z} W_j = \Ltwo$. Indeed, $\left\{ \psi^1_{j,m} \, : \, j,m \in \Z \right\}$ constitutes an orthonormal basis for $\Ltwo$. 

Wavelet bases in higher dimensions can be obtained by taking tensor products of one dimensional scaling functions $\phi^1$ and corresponding wavelets $\psi^1$. In fact, by defining
\begin{align*}
	\phi := \phi^1 \otimes \phi^1, \quad \psi := \phi^1 \otimes \psi^1, \quad \psitilde := \psi^1 \otimes \phi^1,\quad  \breve{ \psi} := \psi^1 \otimes \psi^1,
\end{align*}
one can obtain a multiresolution analysis for $L^2(\R^2)$. More precisely, by denoting the dyadic scaling matrix by
\begin{align*}
	A_{2^j}^{(d)} = \begin{pmatrix} 2^j & 0 \\ 0 & 2^j \end{pmatrix},
\end{align*}
the functions $\left\{ \phi\left(A_{2^j}^{(d)} \cdot - m\right) \, : \, m \in \Z^2 \right \}$ are an orthonormal basis for
\begin{align*}
	V^{2}_j := V_j \otimes V_j.
\end{align*}
Furthermore, $(V^2_j)_{j\in \Z}$ forms a multiresolution analysis for $L^2(\R^2)$. Indeed, $(V^2_j)_{j\in \Z}$ satisfies the defining properties of an MRA
\begin{align*}
	V_j^2 \subseteq V_{j+1}^2, \quad \overline{\bigcup \limits_{j \in \Z} V_j^2} = L^2(\Rtwo), \quad \bigcap \limits_{j \in \Z} V_j^2 = \{0 \},
\end{align*}
and
\begin{align*} 
	f \in V_{j+1}^2 \Leftrightarrow f(2 \cdot, 2 \cdot) \in V_j^2.
\end{align*}
Additionally for $j\in \Z$, the wavelet spaces can be described as follows
\begin{align*}
	W_j^2: = (V_j \otimes W_j) \oplus (W_j \otimes V_j) \oplus (W_j \otimes W_j ).
\end{align*}
In particular, see \cite{Dau}, we have
\begin{align*}
	V_{j+1}^2 = V_j^2 \oplus W_j^2
\end{align*}
and 
$$
\left \{ \psi\left(A_{2^j}^{(d)} \cdot -m \right) \, : \, m \in \Z^2 \right \} \cup \left \{ \psitilde \left(A_{2^j}^{(d)} \cdot -m \right) \, : \, m \in \Z^2 \right \} \cup \left \{ \breve{\psi}\left(A_{2^j}^{(d)} \cdot -m \right) \, : \, m \in \Z^2 \right \}
$$
forms an orthonormal basis for $W_j^2$. Moreover, analogously to the one dimensional case, 
\begin{align*}
	\bigoplus_{j \in \Z} W_j^2 = L^2(\R^2),
\end{align*}
hence,
$$
\left \{ \psi_{j,m} \, : \, j \in \Z, m \in \Z^2 \right \} \cup \left \{ \psitilde_{j,m} \, : \, j \in \Z, m \in \Z^2 \right \} \cup \left \{ \breve{\psi}_{j,m} \, : \, j \in \Z, m \in \Z^2 \right \}
$$
forms an orthonormal basis for $L^2(\R^2)$, where
\begin{align*}
\psi_{j,m} &:= \psi\left(A_{2^j}^{(d)} \cdot -m \right), \quad \psitilde_{j,m} := \psitilde\left(A_{2^j}^{(d)} \cdot -m \right), \quad \breve{\psi}_{j,m} := \breve{\psi}\left(A_{2^j}^{(d)} \cdot -m \right).
\end{align*}
\subsection{Shearlets}\label{sec:Shearlets}
For the definition of shearlet systems we denote the \emph{parabolic scaling matrices} and \emph{shearing matrices} as follows
\begin{align*}
 A_{2^j} := \begin{pmatrix} 2^j & 0 \\ 0 & 2^{\lfloor j/2 \rfloor}\end{pmatrix},\ j \in \N \cup \{ 0 \}, \qquad S_k = \begin{pmatrix} 1 & k \\ 0 & 1 \end{pmatrix},\ k\in \Z.
\end{align*}
Then for a function $\psi \in L^2(\R^2)$ the shearlet elements are defined as
\begin{align*}
 \psi_{j,k,m} = 2^{3j/4 } \psi \left( S_k A_{2^j} \cdot - m \right), \quad k\in \Z, j\in \N \cup \{ 0 \}, m \in \Z^2.
\end{align*}
To adjust for the non-uniform treatment of different directions by the shearing procedure one uses a so-called \emph{cone-adapted discrete shearlet system}.
\begin{deff}\label{Definition:ShearletSystem}
Let $\phi, \psi, \psitilde \in L^2(\R^2)$ be the \emph{generating functions} and $c=(c_1, c_2) \in \R^+\times \R^+$. The associated \emph{(cone-adapted discrete) shearlet system} is defined as
\begin{align*}
	\mathcal{SH}(\phi,\psi, \psitilde, c) = \Phi(\phi, c_1) \cup \Psi (\psi, c) \cup \widetilde{\Psi}(\psitilde,c),
\end{align*}
where
\begin{align*}
	\Phi(\phi,c_1) = \{ \phi( \cdot - c_1m) \, : \, m \in \Z^2\},
\end{align*}
and
\begin{align*}
	\Psi (\psi, c) &= \left\{ \psi_{j,k,m} = 2^{3j/4 } \psi \left( S_k A_{2^j} \cdot - cm \right) \, : \, j \geq 0, |k| \leq 2^{\lfloor j/2 \rfloor}, m  \in \Z^2 \right\},\\
	\widetilde{\Psi} (\psitilde, c) &= \left\{ \psitilde_{j,k,m} = 2^{3j/4 } \psitilde \left( S_k^T \widetilde{A}_{2^j} \cdot - \widetilde{c}m \right) \, : \, j \geq 0, |k| \leq 2^{\lfloor j/2 \rfloor}-1, m  \in \Z^2 \right\},
\end{align*}
where $\widetilde{A}_{2^j} = PA_{2^j}P$ and $\tilde{c} = Pc$ with $P = \begin{pmatrix}
                   0 & 1 \\
                   1 & 0\\
                   \end{pmatrix}$.
\end{deff}
\begin{rem}
In contrast to the \emph{regular cone-adapted shearlet system}, see \cite{KLab3}, we use a cone-adapted shearlet system where the 'diagonal elements', i.e. shearlets which correspond to $|k| = 2^{\lfloor\frac{j}{2}\rfloor}$, are excluded from the second cone $\widetilde{\Psi} (\psitilde, c)$. This system is also considered in \cite{Shearlab}. The reason we choose this system is that the argument for the linear independence between different cones uses the slope of the support shapes, see Figure \ref{fig:twoCones} for an illustration of this approach as well as the proof of Theorem \ref{thm:main}.
\end{rem}

For the remainder of this paper we assume the generators $\phi$ and $\psi$ in Definition \ref{Definition:ShearletSystem} to be \emph{separable}. More precisely, let $\phi^1 \in L^2(\R)$ be a continuous compactly supported scaling function and $\psi^1 \in L^2(\R)$ a corresponding compactly supported wavelet. Then we define the shearlet generators by
\begin{align}
	\phi (x_1,x_2) := \phi^1(x_1) \phi^1(x_2), \quad (x_1,x_2) \in \R^2, \label{eq:generatingPhi}
\end{align}
and
\begin{align}
	\psi (x_1,x_2) := \psi^1(x_1) \phi^1(x_2), \quad (x_1,x_2) \in \R^2. \label{eq:generatingPsi}
\end{align}
The shearlet $\psitilde$ is then defined as $\tilde{\psi} := \psi( P \cdot)$, with $P$ as in Definition \ref{Definition:ShearletSystem}. 
The following theorem was proven in \cite{LimShearlets} and asserts that under some additional standard assumptions the generators chosen as in \eqref{eq:generatingPhi} and \eqref{eq:generatingPsi} yield cone-adapted shearlet frames. We note that the theorem below was proven for the original shearlet system, that is where $|k| = 2^{\lfloor j/2 \rfloor}$ is included in the second cone. However, it also holds for the shearlet system of Definition \ref{Definition:ShearletSystem}.
\begin{theorem}[\cite{LimShearlets}]\label{thm:ShearletFrame}
Let  $\gamma +4 > \alpha > \gamma >4$ and for $(x_1,x_2) \in \R^2$ let
\begin{align*}
	\phi (x_1,x_2) := \phi^1(x_1) \phi^1(x_2), \quad \psi (x_1,x_2) := \psi^1(x_1) \phi^1(x_2), \quad \psitilde(x_1,x_2) &:= \psi(x_2,x_1).
\end{align*}
Further, assume that for almost every $\xi \in \R$
\begin{align*}
	| \widehat{\psi^1}( \xi) | \leq K_1 \frac{|\xi|^\alpha}{(1+ |\xi|^2)^{\gamma/2}}
\end{align*}
and
\begin{align*}
	| \widehat{\phi^1}( \xi) | \leq K_2 \frac{1}{(1+ |\xi|^2)^{\gamma/2}},
\end{align*}
for some $K_1, K_2 >0$. If 
\begin{align}\label{eq:InfPhi}
	\mathop{\mathrm{ess}\inf} \limits_{|\xi| \leq 1/2} | \phihat(\xi)|^2 >0
\end{align}
and
\begin{align}\label{eq:InfPsi}
	\mathop{\mathrm{ess}\inf} \limits_{\beta/2 \leq | \xi | \leq \beta} |\psihat(\xi)|  > 0 \quad \text{ for some } 0 < \beta \leq 1. 
\end{align}
then there exists a sampling parameter $c_0>0$ such that for $c_1=c_2 \leq c_0$ and $c = (c_1, c_2)$ the system $\Scal \Hil ( \phi, \psi, \psitilde, c)$ forms a frame for $L^2(\R^2)$.
\end{theorem}
The assumptions from Theorem \ref{thm:ShearletFrame} can easily be fulfilled, for instance, one can choose $\phi^1$ to be a compactly supported Daubechies MRA scaling function and $\psi^1$ a corresponding wavelet with sufficient decay and vanishing moments. 
According to \cite{Dau}, the Fourier transform of $\phi^1$ obeys the scaling equation
$$\widehat{\phi^1}(\xi) = \widehat{\phi^1}\left(\frac{\xi}{2}\right) m_0\left(\frac{\xi}{2}\right), \quad  \xi \in \R, $$
where
\begin{align*}
	m_0(\xi) = \left(\frac{1+e^{-i \pi \xi}}{2} \right)^N \left(\sum_{s=0}^{N-1} \binom{N-1+s}{s} \sin^{2s}\left( \frac{\pi \xi}{2} \right)\right),
\end{align*}
for some $N \in \N$. 
Then, the scaling functions are given as
\begin{align*}
	\widehat{\phi^1}(\xi) := \prod_{j=1}^\infty m_0\left( \frac{\xi}{2^j} \right).
\end{align*}
Now, it is easy to see that $m_0(\xi) \neq 0$ for all $\xi \in (-1,1)$. Since $\widehat{\phi^1}$ is continuous and $\widehat{\phi^1}(0)=1$, there exists $L\in \N$ such that $|\widehat{\phi^1}\left(\frac{\xi}{2^L}\right)|>0$ for all $\xi \in (-2,2)$. By the scaling relation one obtains
$$\widehat{\phi^1}(\xi) = \widehat{\phi^1}\left(\frac{\xi}{2^L}\right) \prod_{s = 1}^L m_0\left(\frac{\xi}{2^s}\right)\neq 0 \quad \text{ for all }\xi \in (-2,2).$$
Since $|\widehat{\phi^1}|^2$ is continuous we obtain that 
\begin{align*}
	\mathop{\mathrm{ess}\inf} \limits_{|\xi| \leq 1/2} | \widehat{\phi^1}(\xi)|^2 >0,
\end{align*}
which yields that the scaling functions under consideration always obey \eqref{eq:InfPhi}.
Note that the corresponding wavelet $\psi^1$ obeys $\widehat{\psi}(0) = 0$, see \cite{Dau}. Since in our case $\psi^1$ is compactly supported we know that the Fourier transform is analytic and hence $0$ is an isolated root of $\widehat{\psi}^1$. Hence, we can always find $\beta$ such that \eqref{eq:InfPsi} is satisfied.

For the results in this paper to hold we restrict the possible sampling constants to a dense subset of $\R^+\times \R^+$ and also additionally assume linear independence of the underlying wavelet system. This leads to
\begin{deff}
 Let $\phi, \psi$ be as in \eqref{eq:generatingPhi} and \eqref{eq:generatingPsi} and let $c = (c_1,c_2) \in \Q^+\times \Q^+$ with $c_i = a_i/b_i$, where $a_i, b_i\in \N$ and $b_i$ is odd and let the wavelet system 
 $$ \left\{\phi^1(\cdot - c_i m), \psi^1(2^j \cdot - c_i m), j\geq 0 , m\in \Z\right\}$$
be linearly independent for $i = 1,2$.
Then the system $\mathcal{SH}(\phi, \psi, \psitilde, c)$ is called \emph{admissible compactly supported separable shearlet system}.
\end{deff}
Clearly, the restriction of the sampling constant $c$ onto the dense subset does not conflict any characteristic properties, such as the frame property. However, the reader might wonder about the linear independence of the irregular sampled one dimensional wavelet system. Indeed this is not an extreme assumption. First, if this is system is not linearly independent, then the shearlet system $\mathcal{SH}$ can never be linearly independent since for $k =0$ it contains a wavelet system. Second, the linear independence of such irregular wavelet systems well-known problem and has been studied extensively in the history, see \cite{ChrLind,LindnerChristensen} and the references therein.

\section{Linear independence of oversampled compactly supported MRA wavelets}\label{subsec1:mainResults}

The linear independence of compactly supported functions may be studied by investigating the support of the underlying functions. We start with a lemma, that describe the supports of the elements of compactly supported wavelet bases. For the proof we shall use the following common notion for the translated and scaled version of a scaling function $\phi^1$
\begin{align*}
	\phi^1_{j,l}: = \phi^1(2^{j} \cdot - l), \quad j \in \N \cup \{ 0 \}, l \in \Z,
\end{align*}
where we will only consider non-negative scales, since this will be sufficient and also necessary in some of the results presented in this section.
We will work with the minimum of the support of a continuous compactly supported function $f$, which we will denote by $\min \left( \suppp f \right)$.
\begin{lemma}\label{lem:MinSupport}
Let $\phi^1$ be a continuous compactly supported scaling function with $\suppp \phi^1 = [0,s]$ for some  $s \in \R^+$, $(V_j)_{j \in \Z}$ an associated MRA,  and let
$\psi^1$ be a corresponding wavelet. Moreover, let $(W_j)_{j \in \Z}$ denote the wavelet spaces and let $j_0, J\in \N \cup \{0 \}$ with $j_0<J$. 
If $f \in L^2(\R)$ has compact support and 
\begin{align*}
0\neq f\in \bigoplus \limits_{j_0\leq j\leq J} W_j,
\end{align*}
then $f$ is continuous and
\begin{align*}
\min \left( \suppp f \right)\in 2^{-(J+1)}\Z.
\end{align*}
In particular 
\begin{align*}
 \min \left(\suppp f(\cdot - \omega)\right) \in 2^{-(J+1)}\Z + \omega, \quad \text{ for all } \, \omega \in \R.
\end{align*}
\end{lemma}
\begin{proof}
By the MRA structure we have that 
\begin{align*}
 f\in \bigoplus \limits_{j_0\leq j\leq J} W_j \subset V_{J+1}. 
\end{align*}
Since $(\phi^1_{J+1,l})_{l\in \Z}$ is an ONB for $V_{J+1}$ we obtain
\begin{align}
f = \sum \limits_{l\in \Z} \left \langle f, \phi^1_{J+1,l}\right \rangle \phi^1_{J+1,l}. \label{eq:ParsevallOnVJPlus1}
\end{align}
Due to the fact that $f\neq 0$ and $f$, $\phi$ are compactly supported we obtain from \eqref{eq:ParsevallOnVJPlus1} that there exist $l_0, l_1\in \Z$, $l_0<l_1$ such that
\begin{align*}
 \left \langle f, \phi^1_{J+1,l_0}\right \rangle \neq 0 \text{ and }   \left \langle f, \phi^1_{J+1,l}\right \rangle = 0 \text{ for all } l \leq l_0 \text{ or } l \geq l_1,
\end{align*}
hence,
\begin{align*}
f = \sum \limits_{l=l_0}^{l_1} \left \langle f, \phi^1_{J+1,l}\right \rangle \phi^1_{J+1,l}.
\end{align*}
Since the scaling function is continuous, $f$ must also be continuous. Therefore, we can conclude that
\begin{align*}
 \min \left(\suppp f\right) = \min \left(\suppp \phi^1_{J+1,l_0} \right) = 2^{-(J+1)}l_0 \in 2^{-(J+1)}\Z.
\end{align*}
The in particular part of this lemma is clear.
\end{proof}
The information about the minimum of the supports will be used in connection with the following lemma, which states that functions with staggered supports are linearly independent. 
\begin{lem}\label{lem:SupportMethodLemma}
 Let $N \in \N$ and $f_1, \ldots, f_N \in L^2(\R)$ be continuous compactly supported functions and $a_i = \min\left(\suppp f_i\right)  \in \R$ for $i = 1,\dots, N$.
If $ a_i \neq a_j$ for all $1\leq i, j\leq  N$ with $i\neq j$, then the functions $f_1, \ldots, f_N$ are linearly independent. Furthermore, if $\alpha = (\alpha_i)_{i=1}^N\in \C\setminus \{0\}$, then $\min \left( \suppp \sum \alpha_i f_i \right) \in \left\{ a_i \ : \ i = 1\dots N\right\}$. 
\end{lem}

\begin{proof}
W.l.o.g. let $a_1 < \ldots < a_N$, otherwise we reorder the indices.
Let $\lambda_1, \ldots, \lambda_N \in \C$ such that
\begin{align*}
 \sum \limits_{i=1}^N \lambda_i f_i = 0.
\end{align*}
Since $a_1 < a_2 $ we obtain by continuity of $f_1$, that $f_1$ is non-zero on a non-empty interval $I_1 \subset[a_1, a_2)$.
Therefore $\lambda_1$ must be zero and hence $\sum \limits_{i=2}^N \lambda_i f_i = 0$.
Repeating this process leads to $\lambda_1 = \ldots = \lambda_N =0$.
\end{proof}
Lemma \ref{lem:MinSupport} and  Lemma \ref{lem:SupportMethodLemma} are used to prove the next result, which shows that oversampled wavelet systems are linearly independent.
\begin{prop}\label{prop:FiniteLinearIndependentProposition}
Let $(V_j)_{j \in \Z}$ be an MRA with continuous compactly supported scaling function and let $\psi^1$ be a corresponding continuous compactly supported  wavelet. Moreover, let $i=1, \ldots, n$, $J\in \N \cup \{0\}$, $(t_i)_i \in (0,1)$ such that $t_i -t_j \not \in 2^{-J-1}\Z$, for all $i \neq j$. Furthermore for $j=0,...,J$ let $L^i_j\subset \Z$ be of finite cardinality. Then the union of 
\begin{align*}
	\Omega^i_n :=\left\{ \left\{ \psi^1( 2^j (\cdot - t_i) + l \right\}_{l \in L^i_j}\right\}_{j=0}^J , \quad 1 \leq i \leq n
\end{align*}
is linearly independent. 
\end{prop}

\begin{proof}
W.l.o.g. suppose that $\suppp \psi^1 = [0,r]$ for some $r \in \R^+$. Let $f_1, \ldots, f_{n} \in L^2(\R)$ such that
\begin{align*}
0 \neq f_i, \text{ and } f_i \in \spann \Omega_n^i, \quad i = 1, \ldots, n.
\end{align*}
Observe, that since the sets $\Omega_n^i$ have finite cardinality, the functions $f_1, \ldots, f_{n}$ are compactly supported as a finite linear combination of compactly supported functions. Also, since the wavelets are continuous the $f_1, \ldots, f_{n}$ are continuous as well.

\textbf{Claim 1:} The functions $f_1, \ldots, f_{n}$ are linearly independent.

\medskip
Due to Lemma \ref{lem:SupportMethodLemma}, it is sufficient to prove that
\begin{align*}
 \min \left(\suppp f_{i_1}\right) \neq \min \left(\suppp f_{i_2}\right) \quad \text{ for all } i_1 \neq i_2. 
\end{align*}
To this end, let $T_y : L^2(\R) \longrightarrow L^2(\R)$ be the translation operator that maps $f$ to $T_y(f) = f( \cdot + y)$. Then, for any $f^i \in \spann \Omega_n^i$, we clearly have $T_{i/ n} f^i \in \bigoplus_{0 \leq j \leq J} W_j$. By Lemma \ref{lem:MinSupport} we obtain for $i_1,i_2 \in \{1, \ldots, n\}, i_1 \neq i_2$
\begin{align*}
 \min \left(\suppp f_{i_1}\right) \in 2^{-(J+1)}\Z+t_{i_1}  \quad \text{ and } \quad  \min \left(\suppp f_{i_2}\right) \in 2^{-(J+1)}\Z+ t_{i_2}.
\end{align*}
If
\begin{align*}
 \min \left(\suppp f_{i_1}\right) =\min\left( \suppp f_{i_2}\right),
\end{align*}
then there exist $k,s \in \Z$ such that
\begin{align}
 2^{-(J+1)}k+ t_{i_1}  = 2^{-(J+1)}s + t_{i_2}. \label{equation:PropositionContra1}
\end{align}
Multiplying (\ref{equation:PropositionContra1}) by $ 2^{J+1}$ yields
\begin{align*}
 2^{J+1} t_{i_1} + k = 2^{J+1} t_{i_2} + s
\end{align*}
which is equivalent to
\begin{align}
2^{J+1} (t_{i_1}-t_{i_2}) =   s-k.  \label{equation:PropositionContra2}
\end{align}
By the assumptions on $t_i$ we obtain that \eqref{equation:PropositionContra2} cannot be true. Therefore, \eqref{equation:PropositionContra1} is false. This proves Claim 1.

\textbf{Claim 2:} For fixed $i \in \{ 1, \ldots, n \}$ the set of functions 

$$\Omega_n^i = \left\{ \psi^1(2^j (\cdot - t_i) + l) \, : \, l \in L^i_j, j = 0, \ldots, J  \right\}$$ 
is orthogonal.

Since $\{\psi^1 (2^j \cdot - l ) \, : \, l \in \Z, j =0, \ldots, J\}$ are orthogonal by the MRA property and orthogonality remains under a fixed shift operation, $\Omega_n^i$ are orthogonal for fixed $ 1 \leq  i \leq n$. This yields Claim 2. 

For the sake of brevity of notation, we now denote the elements of $\Omega_n^i$ by
\begin{align*}
 \Omega_n^i = \left\{ \psi_{j, l}^{(i)} = \psi^1( 2^j (\cdot - t_i) + l) \, : \, (j,l) \in P_i \right\}, \text{ with } P_i = \{0,\dots, J\} \times L_j^i. 
\end{align*}
for $i = 1, \ldots,  n$. Note that each set $P_i$ is of finite cardinality. It now remains to prove the following statement:

If
\begin{align}
 \sum \limits_{j,l\in P_1} \lambda_{j,l}^{(1)} \psi_{j, l}^{(1)} + \ldots + \sum \limits_{j\in P_n} \lambda_{j,l}^{(n)} \psi_{j, l}^{(n)} = 0, \quad  \lambda_{j,l}^{(i)} \in \C \label{eq:linIndProp6}
\end{align}
then $\lambda_{j,l}^{(i)} = 0$, for all  $(j,l) \in P_i, \ i = 1, \ldots,  n$. 
For this, let us shorten the notation by 
\begin{align*}
\tilde{f}_i: =\sum \limits_{(j,l)\in P_i} \lambda_{j,l}^{(i)} \psi_{j, l}^{(i)}, \quad \text{ for } 1\leq i \leq  n.
\end{align*}
Towards a contradiction we assume, that for some $m \in \{ 1, \ldots,  n \}$ there exists $(j,l) \in P_m$ such that $\lambda_{j,l}^{(m)} \neq 0$.
Consequently, by Claim 2, $\tilde{f}_m \neq 0$. Claim 1 yields linear independence of the $\tilde{f}_i$ and thus 
 \begin{align*}
  \sum \limits_{i=1}^{n} \tilde{f}_i \neq 0,
 \end{align*} 
which contradicts \eqref{eq:linIndProp6}. This finishes the proof.
\end{proof}

Having established the linear independence of oversampled wavelet systems we now aim to prove the independence of shearlet systems.

\section{Linear independence of compactly supported shearlets}\label{subsec2:mainResults}

In this section, we prove the linear independence of $\mathcal{SH}(\phi, \psi, \psitilde, c)= \Phi(\phi, c_1) \cup \Psi(\psi, c) \cup \tilde{\Psi}(\tilde{\psi}, c)$. We proceed by first studying the linear independence of each shearlet cone $\Psi(\psi, c)$ separately.

\begin{theorem}\label{thm:maintheoremConecss}
Let $\mathcal{SH}(\phi, \psi, \psitilde, c) = \Phi(\phi, c_1) \cup \Psi(\psi, c) \cup \tilde{\Psi}(\tilde{\psi}, c)$ be an admissible compactly supported separable shearlet system. Then $\Psi(\psi, c)$ as well as $\tilde{\Psi}(\tilde{\psi}, c)$ are linearly independent. 
\end{theorem}
\begin{proof}
We only show the linear independence of $\Psi(\psi, c)$. The argument for $\tilde{\Psi}(\tilde{\psi}, c)$ is the same with $\psi$ replaced by $\tilde{\psi}$. 
Let $\gamma_1, \dots, \gamma_N \in \Psi(\psi, c)$ such that
\begin{align*}
\gamma_{i}\left(x_1,x_2\right) = 2^{3j_i/4}\psi^1\left(2^{j_i}x_1 + 2^{\left\lfloor\frac{j_i}{2}\right \rfloor}  k_i x_2 + c_1t_i^{(1)}\right) \cdot \phi^1\left(2^{\left\lfloor\frac{j_i}{2}\right \rfloor}x_2 + c_2t_i^{(2)}\right) \quad (x_1,x_2) \in \R^2, i = 1, \ldots, N.
\end{align*}
A priori for $i,l \in \{1, \ldots, N\}$ with $i\neq l$, it is possible that we have $\left(j_i,k_i,t_i^{(1)}\right) = \left(j_l,k_l,t_l^{(1)}\right)$. By defining an equivalence relation $\sim$ such that $i \sim l$, if $\left(j_i,k_i,t_i^{(1)}\right) = \left(j_l,k_l,t_l^{(1)}\right)$, we can write for $K=\{ i\leq N\}/ \sim$ and $L_i:=\{l: i \sim l\}$, 
\begin{align*}
(\gamma_i)_{i=1}^N = \left \{2^{3j_i/4}\psi^1\left(2^{j_i}\cdot + 2^{\left\lfloor\frac{j_i}{2}\right \rfloor}  k_i \cdot + c_1 t_i^{(1)}\right) \cdot \phi^1\left(2^{\left\lfloor\frac{j_{i}}{2}\right \rfloor}\cdot + c_2t_{l_i}^{(2)}\right)  , i\in K, l_i \in L_i \right\}. 
\end{align*}
To obtain linear independence of $(\gamma_i)_{i=1}^N$ we need to show for $\alpha = \left(\alpha_i\right)_{i=1}^N \in \C^N$ that
\begin{align*}
 0 = \sum \limits_{i\in K} \psi^1\left(2^{j_i} \cdot + 2^{\left\lfloor\frac{j_i}{2}\right \rfloor}  k_i \cdot + c_1t_i^{(1)}\right) \cdot \left(\sum\limits_{l_i\in L_i}\alpha_{l_i}    \phi^1\left(2^{\left\lfloor\frac{j_i}{2}\right \rfloor}\cdot + c_2t_{l_i}^{(2)}\right) \right) \implies \alpha = 0.
\end{align*}
By definition we now have that for any $i,l\in K$, with $i\neq l$
\begin{align}
 \left(j_i,k_i, t^{(1)}_i\right)\neq \left(j_l,k_l, t^{(1)}_l\right). \label{eq:simplifyingAssumption2}
\end{align}
Towards a contradiction, we assume $ \alpha \in \C^N\setminus \{ 0 \}$ and define the function 
\begin{align*}
g_{i,\alpha} := \sum\limits_{l_i\in L_i}\alpha_{l_i} \phi^1\left(2^{\left\lfloor\frac{j_{l_i}}{2}\right \rfloor}\cdot + c_2t_{l_i}^{(2)}\right), \quad \text{ for }i\in K.  
\end{align*}
We also set
\begin{align*}
 U(\alpha) = \bigcup \limits_{i\in K} \suppp g_{i,\alpha}.
 \end{align*}
 
\textbf{Case 1:} Assume $U(\alpha)$ is non-empty.

\smallskip
\noindent
Our goal is to show, that
\begin{align}
 \sum \limits_{i\in K} \psi^1\left(2^{j_i} \cdot + 2^{\left\lfloor\frac{j_i}{2}\right \rfloor}  k_i \cdot +c_1 t_i^{(1)}\right) \cdot \left(\sum\limits_{l_i\in L_i}\alpha_{l_i}    \phi^1\left(2^{\left\lfloor\frac{j_i}{2}\right \rfloor}\cdot + c_2t_i^{(2)}\right) \right) \neq 0. \label{eq:TheGoal}
\end{align}
We can construct a finite covering of $U(\alpha)$ subordinate to the supports of $g_{i,\alpha}$ in the following sense.
We pick $Q\in \N$ closed sets $U_q, q = 1\ldots, Q$, that cover $U(\alpha)$ and obey
\begin{align*}
\text{int} \left\{ \suppp g_{i,\alpha} \cap U_q \right \} \neq \emptyset \implies \suppp g_{i,\alpha} \supseteq U_q.
\end{align*}
This can be done by taking a disjoint generator of the algebra generated by the sets $\suppp g_{i,\alpha}$, $i\in K$.

We define corresponding index sets $I_q = \left \{i \in K: \suppp g_{i,\alpha} \cap U_q \neq \emptyset\right \}$. Now, for all $i\in I_q$ we have that $\suppp {g_{i,\alpha}} \supseteq U_q$. For all $i\in K$ we have that $g_{i,\alpha}$ is continuous and hence, the interior of $U(\alpha)$ is non-empty. Since the interior of $U(\alpha)$ is non-empty, there exists a set $U_q$ with non-empty interior. Hence, we can pick 
a point $\hat{x}_2 \in U_q\setminus \Q$ such that $g_{i,\alpha}(\hat{x}_2) \neq 0$. We define  $\tilde{t}_i: = 2^{\lceil\frac{j_i}{2}\rceil}k_i \hat{x}_2 + 2^{-j_i}(c_1 t_i - \lfloor c_1 t_i \rfloor)$ and $\tilde{\alpha}_{i} := g_{i,\alpha}\left(\hat{x}_2\right)\neq 0$ for $i\in I_q$ and obtain
\begin{align}
\sum \limits_{i\in I_q} \psi^1\left(2^{j_i}( \cdot + \tilde{t}_i) +\lfloor c_1 t_i^{(1)}\rfloor\right) \cdot \left(\sum\limits_{l_i\in L_i}\alpha_{l_i} \phi^1\left(2^{\left\lfloor\frac{j_i}{2}\right \rfloor}\hat{x}_2 + c_2t_i^{(2)}\right) \right)= \sum \limits_{i\in I_q} \tilde{\alpha}_{i} \psi^1\left(2^{j_i}( \cdot + \tilde{t}_i)+\lfloor c_1 t_i^{(1)}\rfloor\right). \label{eq:EquiDec}
\end{align}
If we have 
\begin{align}
 (j_{i_1}, \lfloor c_1 t_{i_1}^1\rfloor)  \neq  (j_{i_2}, \lfloor c_1 t_{i_2}^1\rfloor) \text{ or }\tilde{t}_{i_1} - \tilde{t}_{i_2} \not \in 2^{-J-1}\Z \ \text{ for all } i_1, i_2 \in I_q, i_1\neq i_2, \label{eq:indProp3} 
\end{align}
then, since $\tilde{\alpha}\neq 0$, an application of Proposition \ref{prop:FiniteLinearIndependentProposition} to equation \eqref{eq:EquiDec} yields equation \eqref{eq:TheGoal}. 
In order to show that \eqref{eq:indProp3} can be achieved, we observe that we can restrict ourselves to subsets ${I}_q'$ of $I_q$, such that $j= j_s =  j_r $, for $s, r\in {I}_q'$. 
By construction we have that 
 $\tilde{t}_{i_1} -  \tilde{t}_{i_2} \in 2^{-J-1}\Z$ 
is only possible if $c_1 t_{i_1}^1 - \lfloor c_1 t_{i_1}^1\rfloor = c_1 t_{i_2}^1 - \lfloor c_1 t_{i_2}^1\rfloor$. In this case we have by assumption that $k_i = k_j$ and hence with \eqref{eq:simplifyingAssumption2} we have that
$\lfloor c_1 t_{i_1}^1\rfloor \neq \lfloor c_1 t_{i_2}^1\rfloor$.


\textbf{Case 2:} $U(\alpha) = \emptyset$.

\medskip
\noindent
In this case we have that $g_{i,\alpha} = 0$ for all $i\in K$. Recall that for $i\in K$ 
$$
g_{i,\alpha} = \sum\limits_{l_i\in L_i}\alpha_{l_i}{\phi^1\left(2^{\left\lfloor\frac{j_{i}}{2}\right \rfloor}\cdot + c_2t_{l_i}^{(2)}\right)},
$$
and observe that $t_{l_i}^{(2)} \neq t_{k_i}^{(2)}$ for all $l_i, k_i \in L_i, l_i \neq k_i$. Thus the functions $\left(\phi^1\left(2^{\left\lfloor\frac{j_{i}}{2}\right \rfloor}\cdot + c_2 t_{l_i}^{(2)}\right)\right)_{l_i\in L_i}$ are linearly independent by Lemma \ref{lem:SupportMethodLemma}. This implies $\alpha_{l_i} = 0$ for all $l_i \in L_i$ and all $i \in K$. Finally this contradicts the fact that $\alpha \neq 0$ and consequently we will never be in the event of Case 2.
\end{proof}
We have just established the linear independence for shearlets within one cone. Next, we aim for the linear independence of the whole shearlet system.
To this end, we split the shearlet system $\mathcal{SH}(\phi, \psi, \tilde{\psi}, c) = \Phi(\phi, c_1) \cup \Psi(\psi, c) \cup \tilde{\Psi}(\tilde{\psi}, c)$ into three disjoint sets $\Gamma_1$, $\Gamma_2$ and $\Gamma_3$ that are defined by
\begin{align*}
\Gamma_1  := \{\psi_{j,k,m}, \text{ with $k\neq 0$ } \}, \quad \Gamma_2 :=\{\tilde{\psi}_{j,k,m}, \text{ with $k\neq 0$ }\}
\end{align*}
and
\begin{align}
\Gamma_3 &:=  \{\psi_{j,k,m}, \text{ with $k= 0$ } \} \cup \{ \tilde{\psi}_{j,k,m}, \text{ with $k= 0$ } \} \cup \Phi(\phi, c_1).
\end{align}
Since Theorem \ref{thm:maintheoremConecss} implies that $\Gamma_1, \Gamma_2$ are linearly independent we now continue with $\Gamma_3$.
\begin{lem}\label{lem:orthNoShear}
Let $\mathcal{SH}(\phi, \psi, \tilde{\psi}, c) = \Phi(\phi, c_1) \cup \Psi(\psi, c) \cup \tilde{\Psi}(\tilde{\psi}, c)$ be an admissible compactly supported separable shearlet system. Then, $\Gamma_3$ is linearly independent.
\end{lem}
\begin{proof}
The proof uses the linear independence of the one dimensional generator functions. By assumption we have that for a finite set $L\in \Z$ and $J \in \N$ the set
$$
\{\phi^1_m, \psi^1_{j,m}, m\in L, 0 \leq j\leq J\}
$$
is linearly independent. Furthermore by rescaling we obtain that
\begin{align}
\label{eq:linearInd1d}
\{\phi^1_{j_0,m}, \psi^1_{j,m}, m\in L, j_0 \leq j\leq J\}
\end{align}
is linearly independent.

Let $I_1 \subseteq \{0\}$ and $I_2,I_3 \subset \N$ be finite and assume that for every $j \in I_i$ we have finite sets $L_1$ and $L_i^{j} \subset \Z^2$, $i=2,3$, such that with $\alpha^1_m \neq 0$ for all $m\in L_1$ and $\alpha^i_{j,m} \neq 0$ for all $j\in I_i, m\in L_{i}^j$ $i = 1,2,3$ we have

\begin{align} \label{eq:thisShouldBeContradicted}
\sum_{m \in L_1} \alpha_{m}^1 \phi_m + \sum_{j\in I_2, m \in L_2^j} \alpha_{j,m}^2 \psi_{j,m} + \sum_{j\in I_3, m \in L_3^j} \alpha_{j,m}^3 \psitilde_{j,m} = 0.
\end{align}
If we can show that \eqref{eq:thisShouldBeContradicted} implies $ L_1 \cup \bigcup \limits_{j\in I_2} L_2^{j}\cup \bigcup \limits_{j\in I_3} L_3^{j} = \emptyset$, then we obtain linear independence.

Let $j_{max} = \max (I_1 \cup I_2 \cup I_3 )$ and assume that $j_{max} >0$. It will be clear, that the case $j_{max} = 0$ follows similarly. Then $\eqref{eq:thisShouldBeContradicted}$ is equivalent to
\begin{align} 
\sum_{m \in L_1} \alpha_{m}^1 \phi_m + \sum_{j\in I_2\setminus \{j_{max}\}, m \in L_2^j} \alpha_{j,m}^2 \psi_{j,m} + \sum_{j\in I_3\setminus \{j_{max}\}, m \in L_3^j} \alpha_{j,m}^3 \psitilde_{j,m}\nonumber \\
= -\sum_{m \in L_2^{j_{max}}} \alpha_{j_{max},m}^2 \psi_{j_{max},m} - \sum_{m \in L_3^{j_{max}}} \alpha_{j_{max},m}^3 \psitilde_{j_{max},m}.\label{eq:thisShouldBeContradicted2}
\end{align}
Let us distinguish two cases. The first case is that one of the two terms in \eqref{eq:thisShouldBeContradicted2} does not vanish everywhere. We assume w.l.o.g. that $\sum_{m \in L_2^{j_{max}}} \alpha_{j_{max},m}^2 \psi_{j_{max},m} \neq 0$. Then, there exists $\hat{x}_2 \in \R^2$:
\begin{align*}
0 \neq \ &\sum_{m \in L_2^{j_{max}}} \alpha_{j_{max},m}^2 \psi_{j_{max},m}(\cdot,\hat{x}_2)= \  \sum_{m \in L_2^{j_{max}}} \tilde{\alpha}_{j_{max},m}^2 \psi^1_{j_{max},m_1} =: f,
\end{align*}
where $\tilde{\alpha}_{j_{max}, m}^2: = \alpha_{j_{max}, m}^2  \phi^1_{\lfloor\frac{j_{max}}{2}\rfloor,m_2}(\hat{x}_2)$. Note that the second term in \eqref{eq:thisShouldBeContradicted2} is a sum of scaling functions $\phi^1_{\lfloor\frac{j_{max}}{2}\rfloor,m_1}$, when sampled in $x_2$. Furthermore, by the linear independence of the one dimensional functions \eqref{eq:linearInd1d} it is impossible to represent $f$ with functions on lower levels $j < j_{max}$. This contradicts \eqref{eq:thisShouldBeContradicted2}.

The second possibility we need to examine is that both terms of \eqref{eq:thisShouldBeContradicted2} vanish everywhere. Then, for at least one of the terms we have that $L_i^{j_{max}} \neq \emptyset$, $i= 2,3 $. W.l.o.g. we assume that $L_2^{j_{max}} \neq \emptyset$. Then we have that 
\begin{align} 
\label{eq:linearIndependence2D}
\sum_{m \in L_2^{j_{max}}} \alpha_{j_{max},m}^2 \psi_{j_{max},m} = 0.
\end{align}
Using the linear independence of \eqref{eq:linearInd1d}, it is straightforward to see that the functions 
$$
\{\psi_{j_{max},m} \, : \, m\in L_2^{j_{max}}\}
$$
are linearly independent. This implies, that \eqref{eq:linearIndependence2D} cannot hold. Hence we obtain that $L_i^{j_{max}} = \emptyset$, $i= 2,3$ and thus the assumption $j_{max} \geq 0$ cannot hold. Consequently, we obtain that $I_1 \cup I_2 \cup I_3 = \emptyset$. This gives the result.
\end{proof}

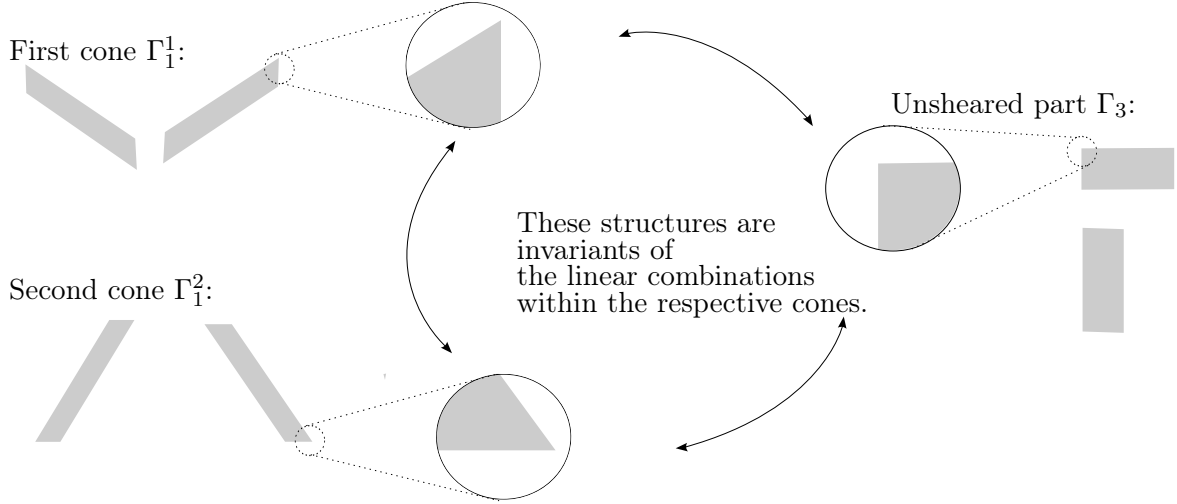
\begin{figure}[htb]
\centering
\includegraphics[width = 0.95\textwidth]{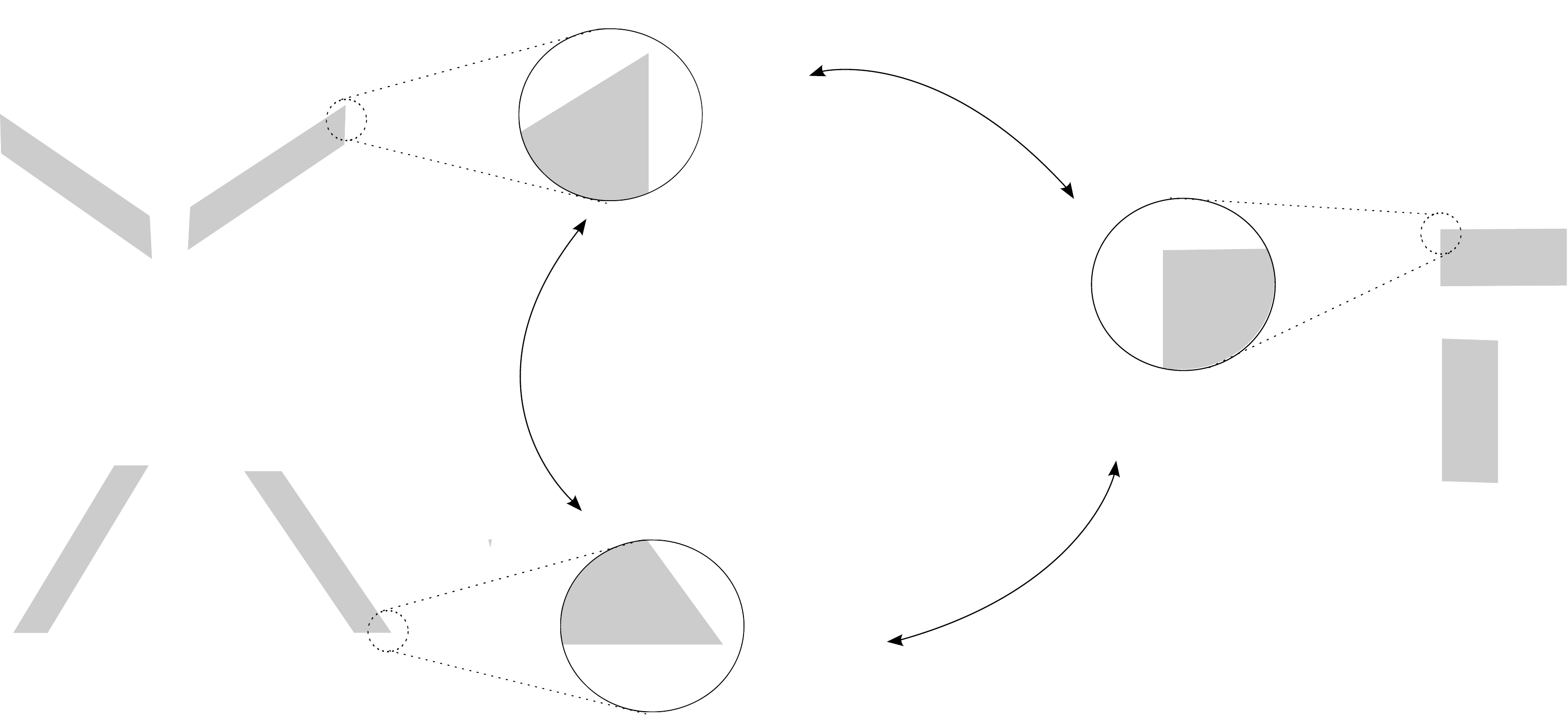}
\label{fig:twoCones}
\put(-250, 105){These structures are}
\put(-250, 95){invariants of }
\put(-250, 85){the linear combinations}
\put(-250, 75){within the respective cones.}
\put(-440, 170){First cone $\Gamma^1_1$:}
\put(-440, 80){Second cone $\Gamma^2_1$:}
\put(-110, 150){Unsheared part $\Gamma_3$:}
\caption{Display of the supports of two shearlet elements from $\Gamma_1$, $\Gamma_2$ and $\Gamma_3$.} \label{fig:supps}
\end{figure}

It turns out, that $\Gamma_1, \Gamma_2$ and $\Gamma_3$ can be distinguished by the shape of the supports of their elements. Furthermore, we will see that certain shape properties stay invariant under finite linear combinations. This observation provides us with the fact that $\spann \Gamma_1$, $\spann \Gamma_2$ and $\spann\Gamma_3$ have trivial intersection. The proof of the following theorem is based on this idea, which is illustrated in Figure \ref{fig:twoCones}.
 
\begin{theorem}\label{thm:main}
Every admissible compactly supported separable shearlet system is linearly independent.
\end{theorem}
\begin{proof}
By Theorem \ref{thm:maintheoremConecss} and Lemma \ref{lem:orthNoShear} we obtain that finite subsets of $\Gamma_1$, $\Gamma_2$, and $\Gamma_3$ are linearly independent.

Thus, we first show
\begin{align}
 \spann \Gamma_1 \cap \spann\Gamma_2 = \{ 0 \}. \label{eq:Gamma11Gamma12}
\end{align}
Let $n = b_1 \cdot b_2$ where $c_i = \frac{a_i}{b_i}, i = 1,2$ satisfy the assumptions of the theorem.
Furthermore, let, without loss of generality, $\suppp \psi^1 = [0,s_1]$ and $\suppp \phi^1 = [0,s_2]$ for some $s_1,s_2\in \R^+$.

We use the following observation that is due to the fact that $W_j \subset V_J$ for all $j < J$. For any $0 \leq j \leq J-1$ and $t \in\Z$ there exists an index set $R_{t}^{j}\subset \Z^2$ and scalars $0\neq\lambda_r \in \C$ with $r\in R_{t}^{j}$, such that
\begin{align}
 \psi^1(2^{j} \cdot - c_1 t) = \sum_{r\in R_{t}^{j}} \lambda_r \phi^1\left(2^{J} \cdot - r/n \right). \label{eq:maintheoremproof1}
\end{align}
We obtain that for every $x_2 \in \R, k\in \Z$
\begin{align*}
\psi^1(2^j \cdot  + 2^{\left \lfloor \frac{j}{2}\right \rfloor} k x_2 - c_1 t) = \sum_{r \in R^j_t} \alpha_r \phi^1\left(2^{J} \cdot  + 2^{J-\left \lceil \frac{j}{2}\right \rceil}k x_2- r/n \right).
\end{align*}
Now, let $f\in (\spann \Gamma_1) \setminus \{0\}$. We assume that the minimal support bound of $f$ in the second variable is $0$, i.e. $f(\cdot, z) = 0$ for all $z<0$. Otherwise this can be achieved by a suitable global shift of $f$.
Then we have
\begin{align*}
f(x_1,x_2) = \sum_{i = 1}^N \alpha_i \psi^1\left(2^{j_i} x_1 + 2^{\frac{j_i}{2}}k_i x_2 + c_1 t_i^{(1)}\right) \phi^1\left(2^{\frac{j_i}{2}}x_2 + c_2 t_i^{(2)}\right), \quad \text{ with $\alpha_i \in \C\setminus \{0\}, N \in \N$}.
\end{align*}
By reordering the indices and a use of equation (\ref{eq:maintheoremproof1}), we can expand any $f \in \spann \Gamma_1$ by
\begin{align}
 f(x_1,x_2) = \sum \limits_{i\in I} \sum_{r_i\in R_i} \alpha_{i,r_i} \phi^1\left(2^{J}x_1 + 2^{J-\left \lceil \frac{j_i}{2}\right \rceil} k_i x_2  - r_i/n \right)\phi^1\left(2^{\frac{j_i}{2}}x_2 + c_2 t_i^{(2)}\right)\label{eq:oneDoesNotRequireNonZeroK}
\end{align}
with $\alpha_{i,r} \neq 0$, $k_i \neq 0$ and $R_i = R^{j_i}_{t_i^{(1)}}$.

For this, we group the indices into the following sets
\begin{align}
	\tilde{R}_{\rho} : = \bigcup_{\left\{i \in I \, : \,\ 2^{J-\left \lceil \frac{j_i}{2}\right \rceil}k_i = {\rho}\right\}} R_i, \quad \rho \in \Z \setminus\{0 \}, |\rho| \leq 2^J.
	\label{eq:groupingOfR}
\end{align}
Then with 
\begin{align}
{\alpha}'_{\rho}(x_2) := \sum_{i \in I : \ 2^{J-\left \lceil \frac{j_i}{2}\right \rceil}k_i = {\rho}}\phi^1\left(2^{\frac{j_i}{2}}x_2 + c_2 t_i^{(2)}\right),
\label{eq:groupingAlpha}
\end{align}
\eqref{eq:oneDoesNotRequireNonZeroK} becomes

\begin{align*}
   \sum_{|\rho| \leq 2^J} \sum_{r_{\rho} \in R_{\rho}}{\alpha'}_{\rho}(x_2) \hat{\alpha}_{{\rho},r_{\rho}} \phi^1 \left(2^{J}x_1 + {\rho} x_2  - r_{\rho}/n \right),
\end{align*}
with $\hat{\alpha}_{\rho,r_\rho} \neq 0$.

If $f \neq 0$ there exists $x_2 \in \R$ such that $f(\cdot, x_2)\neq 0$. Furthermore since $\phi^1$ is continuous we have that for every $x_2$ there exists $\epsilon >0$ such that for all $|\rho|\leq 2^J$ either $\alpha'_{\rho} \neq 0$ for all $\tilde{x}_2 \in B_{\epsilon}(x_2)$ or $\alpha'_{\rho} = 0$ for all $x_2 \in B_{\epsilon}(x_2) \setminus \{x_2\}$.
Hence we know that there exists $\emptyset \neq P \subseteq \{-2^J+1, \dots, 2^J-1\}$ and $\epsilon >0$ such that $\alpha'_{\rho}(\tilde{x}_2)\neq 0$ for $\rho\in P$ and $0<\tilde{x}_2<\epsilon$ and $\alpha'_{\rho} =  0$
for $\rho \not \in P$ and $0<\tilde{x}_2<\epsilon$.

By \eqref{eq:oneDoesNotRequireNonZeroK} for fixed $0<\tilde{x}_2<\epsilon$ we obtain coefficients $\alpha'_{\rho, r_\rho} \neq 0$ for $r_\rho \in R_\rho, \rho \in P$ such that

\begin{align}
f(x_1,x_2) = \sum_{\rho \in P} \sum_{r_{\rho} \in R_{\rho}} {\alpha'}_{{\rho},r_{\rho}} \phi^1 \left(2^{J}x_1 + {\rho} \tilde{x}_2  - r_{\rho}/n \right) \label{eq:OnlyOneSummandDescribesTheLowerSupportBound}.
\end{align}
We now aim to compute the behavior of the lower support bound of $f$ in $x_1$ with the help of the representation obtained in \eqref{eq:OnlyOneSummandDescribesTheLowerSupportBound}. 
Since ${\alpha'}_{\rho,r_\rho} \neq 0$, we can deduce the lower support bound of $f(\cdot,x_2)$ from (\ref{eq:OnlyOneSummandDescribesTheLowerSupportBound}).

First of all, we observe that there is a minimum $r_{min}/n$ of the numbers $r_{\rho}/n $, $r\in R_\rho$, $\rho \in P$. So the lower support bound of the above sum as a function in $x_1$ is given by
\begin{align}
\min \left(\suppp  \sum_{| \rho | \leq 2^J} {\alpha'}_{{\rho},r_{min}} \phi^1 \left(2^{J} \cdot + {\rho} x_2  - r_{min}/n\right)\right) . \label{eq:lowerSuppBound}
\end{align}
Since $x_2>0$ the lower support bound of \eqref{eq:lowerSuppBound} can be found by looking at the unique smallest lower support bound of the respective terms. In fact, if we denote the largest $0<| \rho | \leq 2^J$ such that $\hat{\alpha}_{{\rho},r_{min}}\neq 0$
by $\rho_{max}$ we observe, that the lower support bound of $f(\cdot, x_2)$ is given by 
\begin{align}\label{eq:FoTwenni}
\left(2^{-J}\left(\frac{r_{min}}{n}-{\rho}_{max} x_2\right), x_2\right) \text{ for } 0< x_2< \epsilon.  
\end{align}
Let $\hat{x}_1  = 2^{-J}\frac{r_{min}}{n}$. Assume first, that ${\rho}_{max}>0$. Then for some $\epsilon >0$ the lower support bound of $f(x_1, \cdot)$ is given by
\begin{align*}
\left(x_1, \left(\frac{r_{min}}{n} - 2^{J}x_1\right)/{\rho}_{max} \right) \text{ for } \hat{x}_1 - \epsilon < x_1< \hat{x}_1 .
\end{align*}
However, if ${\rho}_{max}<0$ then locally $\max \left( \suppp f(x_1, \cdot)\right)$ is given by
\begin{align}
\left(x_1, \left(\frac{r_{min}}{n} - 2^{J}x_1\right)/{\rho}_{max} \right) \text{ for } \hat{x}_1 < x_1< \hat{x}_1 + \epsilon. \label{eq:thusSmthgIsLargerNeqZero} 
\end{align}
Now we will see that this behavior of the support bounds is not possible whenever $f \in \Gamma_2$.
%
By the same arguments presented above, we can write $f$ as
\begin{align*}
f(x_1,\cdot) =& \sum \limits_{i\in I} \sum_{r_i\in R_i} \beta_{i,r_i} \phi^1\left(2^{J}\cdot + 2^{J-\left \lceil \frac{j_i}{2}\right \rceil} k_i x_1  - r_i/n \right)\phi^1\left(2^{\lfloor \frac{j_i}{2}\rfloor} x_1 - c_1 t_i^2\right)
\end{align*}

Using the same grouping as in \eqref{eq:groupingOfR} and \eqref{eq:groupingAlpha} we obtain for $|\mu| \leq 2^{J} -1$ the index set $R_\lambda$ and the function $\beta'_{\mu}$.  
Here $\mu$ is smaller than $2^{J}$ because of the assumption on the parameter set for the second shearlet cone, see Definition \ref{Definition:ShearletSystem}.

Again there exists $\emptyset \neq M \subset \{-2^{J}+1,\dots, 2^{J}-1 \}$ such that for some $\epsilon >0$ have that $\beta'_{\mu}(\tilde{x}_1) \neq 0$ for all $\mu \in M$ and $\hat{x}_1 - \epsilon < \tilde{x}_1 < \hat{x}_1$ and for $\mu \not \in M$ we have $\beta'_{\mu}(\tilde{x}_1) = 0$ for all $\hat{x}_1 - \epsilon < \tilde{x}_1 < \hat{x}_1$
Then we obtain that

\begin{align*}
 f(\tilde{x}_1,\cdot) =& \sum_{\mu \in M} \sum_{r_{\mu} \in R_{\mu}} {\beta'}_{{\mu},r_\mu} \phi^1\left(2^{J}\cdot + {\mu} x_1  - r_\mu/n \right).
\end{align*}
with ${\beta'}_{{\mu},r_\mu} \neq 0$

If $\rho_{max}>0$, then by \eqref{eq:thusSmthgIsLargerNeqZero}
\begin{align*}
 \sum_{\mu \in M} \sum_{r_{\mu} \in R_{\mu}} {\beta'}_{{\mu},r_\mu} \phi^1\left(2^{J}\cdot + {\mu} x_1  - r_\mu/n \right) \neq 0
\end{align*}
for all $\hat{x}_1-\epsilon< x_1< \hat{x}_1$. 

Again, there exists $r_{min}' \in \bigcup_{\mu \in \tilde{M}}\tilde{R}_{\mu}$ and a corresponding $\mu_{max}$ as in the first case. Furthermore, 
we obtain from Lemma \ref{lem:SupportMethodLemma} that the minimum support bound in a neighborhood of $\tilde{x}_1$ of $\tilde{f}(z, \cdot)$ is given by
$$\left\{(z, 2^{-J}(r'_{min}/n-\mu_{max} z) \, : \,  z \in B_\epsilon(\tilde{x}_1)  \right \}.$$
Since $|2^{-J}\mu_{max}|<1$ and $|2^{J}/\rho_{max}|\geq 1$, we see that the slope of the lower support bound is different to the previous case. Which implies, that $f$ cannot be in the first and the second cone at the same time.

If $\rho_{max}<0$ then the same arguments yield for an $\tilde{x}_1$ such that $\hat{x}_1  \leq \tilde{x}_1 \leq \hat{x}_1 + \epsilon$ that the lower support bound of $\tilde{f}(\tilde{x}_1, \cdot)$ is given by 
$$\left\{(z, 2^{-J}(r'_{min}/n -\mu_{max}z)) \, : \,  z \in B_\epsilon(\tilde{x}_1) \right \}.$$
If furthermore $\mu_{max}>0$ we obtain that the lower support bound of $\tilde{f}(\cdot, \hat{x}_2)$ for $\hat{x}_2 = 2^{-J}(r'_{min}/n -\mu_{max}\tilde{x}_1)$ is locally given as
$$((r'_{min}/n-2^{J} \tilde{x}_2)/\mu_{max}, \tilde{x}_2), \text{ for } \tilde{x}_2 \text{ in a neighborhood of } B_\epsilon(\hat{x}_2),$$
which contradicts \eqref{eq:FoTwenni}

Lastly, if $\rho_{max}<0$ and $\mu_{max}<0$, then $2^{-J}(r'_{min}/n -\mu_{max}\tilde{x}_1)<0$ which cannot happen since we assumed, that the smallest $x_2$ such that $f(\cdot, x_2) \neq 0$ is $0$.

Hence $\spann \Gamma_1 \cap \spann \Gamma_2 = {0}$. Furthermore, $\spann \Gamma_1 \cap \spann \Gamma_3 = {0}$, since for functions from $\Gamma_3$ the lower support bounds along slices remain constant on small intervals in contrast to functions from $\Gamma_1$ or $\Gamma_2$, see Figure \ref{fig:supps}. 

Finally, the behavior of lower supports in \eqref{eq:lowerSuppBound} remain unchanged if we assume $f\in \spann \left(\Gamma_1 \cup \Gamma_3\right)$ with the exception, that $\rho_{min}$ can now be $0$. In that case the lower support bound can also remain constant, which, as we have seen, does not happen for functions in $\spann \Gamma_2$. 
\end{proof}
Having established the linear independence of shearlet systems we now comment on the property of $\omega$-independence.
\section{Beyond linear independence}\label{sec:OmegaInd}

Theorem \ref{thm:main} shows that every admissible compactly supported separable shearlet system is linearly independent. Definition \ref{def:linInd} raises the question whether the stronger property of $\omega$-independence can also be fulfilled for every such shearlet system. In fact, we can make a rather simple observation to convince ourselves that $\omega$-independence cannot hold for whole class of compactly supported separable shearlet systems described above. This can be seen by the following argument. 

Let $\mathcal{SH}(\phi, \psi, \psitilde, c)$ be an admissible compactly supported separable shearlet system. Then it is clear, that $\mathcal{SH}(\phi, \psi, \psitilde, c)\subsetneq \mathcal{SH}(\phi, \psi, \psitilde, c/n)$ for all $n\in \N\setminus \{1 \}$. If $c$ is chosen in accordance with Theorem \ref{thm:ShearletFrame}, then $\mathcal{SH}(\phi, \psi, \psitilde, c)$ and $\mathcal{SH}(\phi, \psi, \psitilde, c/n)$ both form frames for $L^2(\R^2)$ and hence span $L^2(\R^2)$. However, since $\mathcal{SH}(\phi, \psi, \psitilde, c/n)$ is overcomplete, it can never be $\omega$-independent. This means that linear independence is the best we can hope for with the class of shearlet systems discussed in this work. Summarizing, this yields
\begin{prop}
There exist admissible compactly supported separable shearlet systems that are not $\omega$-independent.
\end{prop}

Another question is whether shearlets can constitute Riesz bases. This problem could be tackled by finding a splitting into Riesz sequences in the spirit of the Feichtinger conjecture. In particular, Theorem \ref{thm:main} allows us to conclude that for every finite subset $I\subset \N$ the function $(\psi_i)_{i \in I}$ fulfill the Riesz property, i.e.there exist $0<A_I \leq B_I < \infty$ such that:
\begin{align}
A_I \|(c_i)_{i\in I}\|_{\ell^2}^2 \leq \left\|\sum \limits_{i\in I} c_i \psi_i\right\| \leq B_I \|(c_i)_{i\in I}\|_{\ell^2}^2 \quad \text{ for all } (c_i)_{i\in I}\in \ell^2(I). \label{eq:RieszBound}
\end{align}
It is easy to see, that the upper bound in \eqref{eq:RieszBound} can be chosen to be independent of the index set $I$ if the whole system $(\psi_i)_{i \in \N}$ is a frame. Indeed, $B$ can be chosen as the upper frame bound of $(\psi_i)_{i\in\N}$. Hence, the crucial part is the lower Riesz bound. In fact, by our result the extraction of Riesz sequences can now also be studied by analyzing the behavior of the largest lower frame bound in \eqref{eq:RieszBound}. 

For instance, let $I \subset \N$ with $| I | = \infty$ and $I_1 \subset I_2 \subset \ldots  \subset I$, with $\bigcup_{i\in \N} I_i = I$, one could study the lower frame bounds  $A_{I_N}$ of the \emph{frame sequence} $(\psi_i)_{i \in I_N}$, i.e. frame for its span. The reason for this is the following result that can be found in \cite{LindnerChristensen}, see also \cite{Chr}
\begin{lemma}\label{lemma:RieszLinInd}
Let $I$ be a countable index set and $(\psi_i)_{i \in I} \subset \Hil$ a frame for $\mathcal{X} := \overline{\spann} \{ \psi_i \, : \, i \in I \} \subset \Hil$. Let $I_1 \subset I_2 \subset \ldots  \subset I$ such that $\bigcup_{i\in \N} I_i = I$ and let $A_N$ denote the optimal, i.e. largest, lower frame bound for the frame sequence $(\psi_i)_{i \in I_N}$. Then the following are equivalent
\begin{itemize}
	\item[i)] $(\psi_i)_{i \in I}$ is a Riesz basis.
	\item[ii)] $(\psi_i)_{i \in I}$ is $\omega$-independent. 
	\item[iii)] $(\psi_i)_{i \in I}$ is linearly independent and $\inf_{N \in \N} A_N>0$.
\end{itemize}
\end{lemma}
If $(\psi_i)_{i \in \N}$ denotes the shearlet frame, then, due to Lemma \ref{lemma:RieszLinInd} one can search for $I \subset \N$ with $|I| = \infty$ such that $A_I := \inf_{N \in \N} A_N>0$ where $A_N$ are the optimal lower frame bounds of the corresponding frame (for its span)$(\psi_i)_{i \in I_N}$ and $I_1 \subset I_2 \ldots \subset I$, $\bigcup_{i\in \N} I_i = I$.  

The question about the existence of $\omega$-independent shearlet frames, or equivalently shearlet Riesz bases is a delicate question for future work. 

\section{Related systems}

\subsection{Band-limited shearlets and curvelets}

Given $N$ band-limited functions $f_i,\ i=1, \dots, N$ and $M$ points $t_j$ taken from any lattice $\alpha \Z \times \beta \Z$, $\alpha, \beta \in \R$, the linear independence of the functions 

$$
\{f_{i,j}: = f_{i}(\cdot - t_j), i=1, \dots N, j = i \dots M\},
$$
can be examined as follows. Since 
$$
\sum a_{j} \hat{f}_{i,j} = \hat{f}_{i} \sum a_{j} e^{-i \langle \cdot, t_j\rangle},
$$
the support of 
$\sum_{j\leq M} a_{i,j}f_{i, j}$ in frequency is either empty or the same as that of $f_i$.

Furthermore, translates of $L^2$ functions with translations on a lattice are always linearly independent \cite{Linnell}, hence for fixed $i\leq N$ the functions $f_{i, j}$ are linearly independent and so,
$\suppp \sum_{j\leq M} a_{i,j}f_{i, j} = \emptyset$ only if $a_{i,j} = 0$ for all $j \leq M$.
If the supports in frequency cover the frequency plane in such a way, that for a finite number of $f_i$ there is always a point $\xi$ such that $f_i(\xi) \neq 0$ for only one $i\leq N$, then the linear independence can be obtained as a direct consequence. This can yield linear independence for band-limited curvelets \cite{CanDon} as well as band-limited shearlets \cite{GuoKutLab2006, GuoLabLimWeiWil} depending on the generators.

\subsection{Compactly supported curvelets}

In \cite{RasNie} the authors introduced a machinery to construct compactly supported curvelet-type systems which share the almost optimally sparse approximation property of usual curvelets \cite{CanDon} and constitute a frame. Since shearlet systems and curvelet systems are closely related it is natural to ask whether our results can be carried over to the compactly supported curvelet-type systems introduced in \cite{RasNie}. The methods we used in this paper will not be applicable to these curvelet systems. This is due to several reasons: First, it is fundamental for our results to hold that the directional operation acts as a shift in one direction along slices. This is indeed the case for shearing but not for rotation. Second, compactly supported shearlet systems can be build from a 1D MRA scaling function and corresponding wavelets that give rise to single generators for each cone whereas curvelets do not provide single generators. Both properties combined allow us to view shearlets along one direction 
as a shifted version of an MRA wavelet and hence allows the use of typical characteristics of such MRA.

\section{Acknowledgments} 
The authors would like to thank G. Kutyniok and W.-Q Lim for useful discussions. J. Ma acknowledges support from the Berlin Mathematical School. J. Ma and P. Petersen are supported by the DFG Collaborative Research Center TRR 109 "Discretization in Geometry and Dynamics".

\bibliographystyle{abbrv}
\bibliography{LinIndRef}
\end{document}